\documentclass[10pt,english]{amsart}

\usepackage[colorlinks=true, pdfstartview=FitV, linkcolor=blue,
citecolor=blue, urlcolor=blue]{hyperref}

\usepackage{color}
\usepackage{calc}

\usepackage{ifthen}
\usepackage{amsfonts}
\usepackage{amsthm}
\usepackage{amstext}
\usepackage{enumitem}
\usepackage{graphicx}

\numberwithin{equation}{section}
\newtheorem{theorem}{Theorem}[section]

\newtheorem{corollary}[theorem]{Corollary}

\theoremstyle{definition}

\newtheorem{remark}[theorem]{Remark}

\newlength{\breite}
\newlength{\hbreite}

\DeclareMathOperator{\en}{W}
\DeclareMathOperator{\wen}{W}
\DeclareMathOperator{\grad}{\nabla_{L^2} \en}
\DeclareMathOperator{\wgrad}{\nabla_{L^2} \wen}
\DeclareMathOperator{\diam}{diam}
\DeclareMathOperator{\vol}{vol}
\DeclareMathOperator{\ar}{\mu}

\begin{document}

\title[Finite time singularities]{A note on singularities in finite time for the constrained Willmore flow}
\author{
Simon Blatt}
\thanks{The author was supported through FWF grant "The gradient flow of curvature energies"}
\address[Simon Blatt]{Paris Lodron Universit\"at Salzburg, Hellbrunner Strasse 34, 5020 Salzburg, Austria}
\email{simon.blatt@sbg.ac.at}
\date{\today}

\begin{abstract}
This work investigates the formation of singularities under the steepest descent $L^2$-gradient flow of the functional $\en_{\lambda_1, \lambda_2}$, the sum of the Willmore energy, $\lambda_1$ times the area, and $\lambda_2$ times the signed volume of an immersed closed surface without boundary in $\mathbb R^3$.
We show that in the case that $\lambda_1>1$ and $\lambda_2=0$ any immersion develops singularities in finite time under this flow. If $\lambda_1 >0$ and $\lambda_2 > 0$,  embedded closed surfaces with energy less than 
$$8\pi+\min\{(16 \pi \lambda_1^3)/(3\lambda_2^2), 8\pi\}$$  and positive volume evolve singularities in finite time. If in this case the initial surface is a topological sphere and the initial energy is less than $8 \pi$, the flow shrinks to a round point in finite time. We furthermore discuss similar results for the case that $\lambda_2$ is negative. 
%Using a new reverse isoperimetric inequality, we show that also in the case $\lambda_1 =0$ and $\lambda_2 >0$ less than $8\pi$  and positive volume converge to a round point in finite time.
 
These results strengthen the ones of McCoy and Wheeler in \cite{McCoy2016}. For $\lambda_1 >0$ and $\lambda_2 \geq 0$ they showed that embedded closed spheres with positive volume and energy close to $4\pi$, i.e.  close to the Willmore energy of a round sphere, converge to round points in finite time.
%
%We also present a kind of reverse isoperimetric inequality, which allows us to extend the result above to the case $\lambda_1 =0$ but $\lambda_2 >0.$
\end{abstract}

\maketitle

\tableofcontents

\section{Introduction} \label{sec:Int}

In \cite{Helfrich1973}, Helfrich suggested the functional 
$$
 \en^{H_0}_{\lambda_1, \lambda_2} (f) := \int_{\Sigma} |H_f+H_0|^2 d\mu_f + \lambda_1 \ar (f) + \lambda_2 \vol(f)
$$
for immersions $f:\Sigma \rightarrow \mathbb R ^3$ of two-dimensional compact connected surfaces $\Sigma$ without boundary to study lipid bilayers. Here, $H_f= \frac 1 2 (\kappa_1 + \kappa_2)$ denotes the mean curvature and $\ar_f$ the surface measure on $\Sigma$ induced by $f$. The constant $H_0 \in \mathbb R$ is called spontaneous curvature and $\vol(f)$ denotes the signed inclosed volume given by
$$
 \vol(f) = \int_{[0,1] \times \Sigma} \phi_f^*(dvol),
$$
where $ \phi^*(dvol)$ denotes the pull-back of the standard volume form $dvol = dx^1 \wedge dx^2 \wedge dx^3$ on $\mathbb R^3$ under $\phi_f: [0,1] \times \Sigma \rightarrow \mathbb R^3, \phi_f(t,x) := tf(x)$.
 
Helfrich found that these energies are well suited to explain the characteristic shape of red blood cells: the shape of a biconcave disk. In honor of his work, the functional $\en^{H^0}_{\lambda_1, \lambda_2}$ is now called \emph{Helfrich functional}.

In this article, we will restrict our attention to the special case of zero spontaneous curvature $H_0$. We only consider
\begin{equation} \label{eq:Helfrich}
 \en_{\lambda_1, \lambda_2}(f):=\en^{0}_{\lambda_1, \lambda_2} = \int_{\Sigma} |H_f|^2 d\ar_f + \lambda_1 \ar (f) + \lambda_2 \vol(f),
\end{equation}
in which case the first summand is the \emph{Willmore energy}
\begin{equation} \label{eq:Willmore}
 \wen(f) := \int_{\Sigma} |H_f|^2 d\mu_f 
\end{equation}
We only deal with this case since the Willmore functional is scale invariant - indeed already Blaschke \cite{Blaschke1929} observed that it is indeed invariant under M\"obius transformations that leave the surface bounded. Willmore proved that the Willmore energy is always greater or equal to $4 \pi$ with equality only for a parametrized round sphere.
 
Helfrich calculated the $L^2$-gradient of $\en_{\lambda_1, \lambda_2}$. It is known to be equal to 
\begin{equation}\label{eq:HelfrichGradient}
 \grad_{\lambda_1, \lambda_2}^{H_0} (f) = (\Delta_f H_f + 2 (H_f+H_0)(H_f^2-H_0H_f -  K_f) - \lambda_1 H_f - \lambda_2) \nu_f,
\end{equation}
where $\nu_f$ denotes the unit normal along $f$, $K_f$ the Gau\ss{} curvature, and $\Delta_f$ the Lapace-Beltrami operator. In the case of zero spontaneous curvature this reads as
\begin{equation}\label{eq:ConstrainedWillmoreGradient}
 \grad_{\lambda_1, \lambda_2}(f) = \grad_{\lambda_1, \lambda_2}^{0} (f) = (\Delta_f H_f + 2 H_f (H_f^2-  K_f) - \lambda_1 H_f + \lambda_2) \nu_f.
\end{equation}

We will consider smooth families of smooth immersions $f_t: \Sigma \rightarrow \mathbb R^3$, $t \in [0,T)$ of a compact surface $\Sigma$ without boundary that are solutions to the steepest $L^2$-gradient flow of the Helfrich functional with zero spontaneous curvature $H_0 =0$, i.e. that solve
\begin{equation} \label{eq:ConstrainedWillmoreFlow}
\partial_t f_t = - \grad_{\lambda_1, \lambda_2} (f_t) \quad \forall t \in [0,T).
\end{equation}
Note that such a family of immersions satisfies the equality
\begin{equation} \label{eq:EnEq}
 \frac d {dt} \en_{\lambda_1, \lambda_2} (f_t) = - \|\grad_{\lambda_1, \lambda_2}(f_t)\|^2_{L^2(d\mu_f)}.
\end{equation}

 Let us state the following short time existence theorem which is an immediate consequence of Theorem 1.1 in \cite{Mantegazza2011} the proof of which was based on the Lions-Lax-Milgram theorem and an analysis of the linearized problem by Polden \cite{Huisken1999}.

\begin{theorem}[cf. \cite{Mantegazza2011}]
 Suppose $f_0: \Sigma \rightarrow \mathbb R ^3$ is a compact smoothly immersed surface without boundary. There exists a unique maximal smooth family of immersions $f:\Sigma \times [0,T) \rightarrow \mathbb  R^3$ solving \eqref{eq:ConstrainedWillmoreFlow} with $f(\cdot, 0) = f_0.$
\end{theorem}

For the case of the gradient flow of $\int_{\Sigma} |H_f-H_0|^2 d\mu_f$ constrained to immersions of fixed length and volume, short time existence was proven by Kohsaka and Nagasawa. Note that short time existence of the flow above can for example also be derived using analytic semi-groups as carried out for the Willmore flow and the surface diffusion flow in \cite{Escher1998,Mayer2003}. Yannan Liu \cite{Liu2012} could bound the lifespan of the flow from below if there is only a small quantum of energy within balls of a given scale thus extending the corresponding result of Kuwert and Sch\"atzle for the Willmore flow \cite{Kuwert2002}.

The following two theorems summarize the results of this article. We set for $\lambda_2 \not= 0$
\begin{equation}\begin{aligned}
\varepsilon_1(\lambda_1, \lambda_2)  &= \frac {\lambda_1^2} {\lambda_2^2}  16 \pi \\
\varepsilon_2 (\lambda_1, \lambda_2)  &= \frac { \lambda_1^3} {3 \lambda_2^2} 16 \pi.
\end{aligned}\end{equation}
Then the following theorem on existence of finite time singularities holds.

\begin{theorem} \label{thm:FormationOfSingularities}
Let us assume that $\lambda_1 >0$ and let $\Sigma$ be a smooth compact manifold without boundary.
\begin{enumerate}[label={(\roman*)}]
\item If $\lambda_2 =0$, the locally constrained Willmore flow \eqref{eq:ConstrainedWillmoreFlow} starting with any initial immersed surface $f_0:\Sigma \rightarrow \mathbb R^3$ forms a singularity in finite time.
\item If $\lambda_2 \not= 0$, the flow starting with an embedding $f_0: \Sigma \rightarrow \mathbb R^3$ with $\vol(f)>0$ forms a singularity in finite time if either 
  \begin{quote}
  $\ar(f_0) < \varepsilon_1(\lambda_1, \lambda_2) $ and $
  \en_{\lambda_1, \lambda_2} (f_0) < 4 \pi + \varepsilon_2 (\lambda_1, \lambda_2) $
  \end{quote}
  or
  \begin{quote}
   $\lambda_2 > 0$ and   $ 
  \en_{\lambda_1, \lambda_2} (f_0) < 8 \pi + \min\{\varepsilon_2 (\lambda_1, \lambda_2) ,8\pi\}.
  $
  \end{quote} 
\end{enumerate}
\end{theorem}

This substantially improves the result of McCoy and Wheeler in \cite{McCoy2016} who only considered the case $\lambda_2 \geq 0$, $\vol (f_0) \geq 0$ and had to assume that $\en_{\lambda_1, \lambda_2}  (f_0)< 4 \pi + \varepsilon$ for a universal constant $\varepsilon>0$, in principle calculable but not further specified. Furthermore our method of proof is based on the different scaling of the three summands that make up the energy and does not rely on sophisticated energy estimates as the proof of the corresponding result by McCoy and Wheeler.

For the case of the constrained Willmore flow of spheres, i.e in the case that $\Sigma = \mathbb S^2$, we can also show that the flow sub-converges to a round point in the sense that there are times $t_j \rightarrow T$ and a point $x\in \mathbb R^n$ such that the rescaled immersions
$$
 \left(\frac {4\pi} {\ar(f_{t_j})} \right)^{\frac 1 2} (f_{t_j} - x)
$$
converge to the unit sphere. More precisely we have

\begin{theorem} \label{thm:RoundPoints}
If in the situation of Theorem \ref{thm:FormationOfSingularities}  furthermore $\Sigma = \mathbb S ^2$ and $$\lim_{t\uparrow T} \en_{\lambda_1, \lambda_2}(f_t) < 8 \pi,$$ then the flow sub-converges to a round point in finite time $T<\infty$.
\end{theorem}

This result builds on the contruction of a blowup profile in \cite{McCoy2016}.
We show at the end of Section \ref{sec:Blo} that the constant $8\pi$ Theorem~\ref{thm:RoundPoints} is sharp.

Let us finish this introduction with an outline of the structure of the remaining article. In Section \ref{sec:Rev} we shortly recapitulate the findings of McCoy and Wheeler on both the existence of critical points of the Helfrich functional as well as on the formation of singularities in finite time. Section \ref{sec:Exi} shows how to use the scaling properties of the different summands that make up the energy $\en_{\lambda_1, \lambda_2}$ to show that the flow \eqref{eq:ConstrainedWillmoreFlow} develops singularities in finite time.  In Section \ref{sec:Blo} we recapitulate the construction of a blowup profile at a point singularity from \cite{Kuwert2001} and \cite{McCoy2016}. We again use the different scalings of the components to show that this blowup profile must be a Willmore surfaces without assuming any initial bound on the energy of the initial surface (cf. Theorem~\ref{thm:BlowupWillmore}). Combining this with the point removability results of Kuwert and Sch\"atzle \cite{Kuwert2004} and the classification result for Willmore spheres in \cite{Bryant1984} we can show the convergence to round points as stated in Theorem~\ref{thm:RoundPoints}.

\section{Review of the results of McCoy and Wheeler} \label{sec:Rev}

In the pioneering paper \cite{McCoy2013}, McCoy and Wheeler completely classified all critical immersions of the functional $\en_{\lambda_1, \lambda_2}^{H_0}$ of complete surface without boundary under the assumption that 
\begin{equation}\label{eq:NearSpheres}
\wen(f) \leq 4 \pi + \varepsilon_0
\end{equation}
where the constant $\varepsilon_0 >0$ can in principle be deduced from the constructive nature of their proof.  Note that since the Willmore energy of a round sphere is equal to $4 \pi$, \eqref{eq:NearSpheres} in a certain way says that the distance of $f_0$ from parametrizing a round sphere is small.
They showed
 
\begin{theorem}[\protect{\cite[Theorem 1]{McCoy2013}}] \label{thm:McCoy2013}
There is an absolute constant $\varepsilon_0>0$ such that the following holds: Suppose that $f:\Sigma \rightarrow \mathbb R^3$ is a smooth properly immersed complete surface without boundary and
\begin{equation}
 \int_{\Sigma} \| A^0\|^2 d\mu < \varepsilon_0.
\end{equation}
Then the following is true:
If $\lambda_1 >0$ 
\begin{enumerate}
 \item[($\lambda_2 < 0$)] $\grad_{\lambda_1, \lambda_2} (f) =0$ if and only if $f(\Sigma)$  is a sphere of radius $-\frac{2\lambda_1}{\lambda_2}$,
 \item[($\lambda_2 =0$)] $\grad_{\lambda_1, \lambda_2} (f) =0$ if and only if $f(\Sigma)$  is a  plane,
 \item[($\lambda_2 >0$)] $\grad_{\lambda_1, \lambda_2} (f) \not=0$.
\end{enumerate}
\end{theorem}

It is a straightforward calculation that spheres of radius $r_0$ remain spheres under the flow where the radii satisfy
$$
 \partial_t r = -\lambda_1 \frac 1r - \lambda_2.
$$

So if $\lambda_1$ and $\lambda_2$ or both non-negative and at least one of them is different from zero, the solution to \eqref{eq:ConstrainedWillmoreFlow} sub-converges to a round point
in finite time.
For the case that $\lambda_1 >0$, $\lambda_2 \geq 0$,  McCoy and Wheeler could extend the above result to the following statement about solutions to the gradient flow \eqref{eq:ConstrainedWillmoreFlow}.

\begin{theorem} [\cite{McCoy2016}] \label{thm:McCoy2016} 
 There is an $\varepsilon>0$ such that initial immersions $f_0 : \Sigma \rightarrow \mathbb R^3 $ of a compact manifold without boundary $\Sigma$, with positive signed inclosed volume and
 $$
  \en_{\lambda_1, \lambda_2}(f_0) \leq 4 \pi + \varepsilon_0
 $$
 sub-converge to a round point under the evolution equation \eqref{eq:ConstrainedWillmoreFlow}
\end{theorem}

One of the main ingredients to the proof of their theorem above, is the highly non-trivial fact that under the condition \eqref{eq:NearSpheres}
we have
$$
 \|\grad_{\lambda_1, \lambda_2}\|^2_{L^2(\mu_f)} \geq c_0 \|\wgrad\|^2_{L^2(\mu_f)}
$$
for some $c_0 >0$ which they prove using sophisticated energy estimates. We will not make use of this estimate at all. Furthermore, their method did not allow them to treat the case $\lambda_2 <0$ which we can. Theorem \ref{thm:FormationOfSingularities} reduces the gap between the existence or better non-existence of critical points in Theorem~\ref{thm:McCoy2013} and the existence of finite time singularities in Theorem~\ref{thm:McCoy2016}.

\section{Existence of finite time singularities - an approach based on scaling} \label{sec:Exi}

Let us now prove the main results of this article on the existence of finite time singularities of the locally constrained Willmore flow. All these results are based on the different scaling behavior of the three terms building $\en_{\lambda_1, \lambda_2}.$ Apart from this, we will use that, by an inequality of Yau \cite{Li1982}, we have
$$
  4 \pi \cdot \#f^{-1}(x) \leq \en(f).
$$
So especially $\en(f) < 8 \pi$ implies that $f$ is an embedding. Peter Topping \cite[Lemma 1]{Topping1998} showed that
\begin{equation} \label{eq:EstDiam}
\diam(f) \leq \frac {2}{\pi} \sqrt{\ar(f) \wen(f)}.
\end{equation}
This estimate is a sharpened version of \cite[Lemma 1.1]{Simon1993}.
  
\begin{theorem} \label{thm:FiniteTime1}
Let $\lambda_1>0$  and $f_0: \Sigma \rightarrow \mathbb R^3$ be a compact smoothly immersed surface without boundary. Let us further assume that either $\lambda_2 =0$ or  
$ \lambda_2 >0$ and $f_0$ be an embedding with positive enclosed volume nad $\en_{\lambda_1, \lambda_2}(f_0) \leq 8 \pi$ Then the maximal time of smooth existence $T$ for the constrained Willmore flow with initial data $f_0$ satisfies
$$
 T \leq  \frac { \left( \en_{\lambda_1,\lambda_2}(f_0) \right)^2 - (4 \pi)^2 }{2 \pi^2 \lambda_1^2}.
$$
\end{theorem}

\begin{proof}
Let us start with an observation for a general immersion $f:\Sigma \rightarrow \mathbb R^3$. We consider the dilations 
$$ f_\alpha = \alpha (f-p)$$ around a fixed point $p$ in $f(\Sigma)$. Note that
the Willmore energy stays constant while the area behaves like $\alpha^2 \ar(f_1)$ and the volume goes like $\alpha^3 \vol(f_1)$. 
By the definition of the $L^2$ gradient we hence find
\begin{equation} \label{eq:Scaling}
 \int_{\Sigma} \grad_{\lambda_1, \lambda_2} (f) (f-p) d\mu_f= 
\frac {d} {d\alpha} \left.\left(\en_{\lambda_1, \lambda_2}(f_\alpha) \right) \right|_{\alpha=1} =  2 \lambda_1 \ar(f) + 3 \lambda_2 \vol(f).
\end{equation}
Using the Cauchy-Schwartz inequality together with \eqref{eq:EstDiam}, we have
\begin{multline}\label{eq:EstimateGradient}
  \int_{\Sigma} \grad_{\lambda_1, \lambda_2} (f) (f-p) d\mu \leq \|\grad_{\lambda_1, \lambda_2} (f)\|_{L^2(\mu)}\cdot \|f-p\|_{L^2(\mu)} \\ 
  \\ \leq \|\grad_{\lambda_1, \lambda_2} (f)\|_{L^2(\mu)}  \diam(f) \sqrt{\ar(f)}  \\
  \leq \frac 2 {\pi} \sqrt{\en(f)}\|\grad_{\lambda_1, \lambda_2} (f)\|_{L^2(\mu)} \mu(f).
\end{multline}
Combining equation \eqref{eq:Scaling} with the estimate \eqref{eq:EstimateGradient}, we get, if $ \lambda_2 \vol(f)\geq 0$, 
\begin{equation} \label{eq:EstGrad}
 \|\grad_{\lambda_1, \lambda_2} (f) \|^2_{L^2(\mu)} \geq  \frac {\pi^2 \lambda^2_1}{ \wen(f)}.
\end{equation}

Let us now consider the unique maximal solution $f_t$, $t \in [0,T)$, of the constrained Willmore flow \eqref{eq:ConstrainedWillmoreFlow} with initial data $f_0$ as in the statement of the theorem. Under the assumption that $\lambda_2 \vol(f)$ stays non-negative along the flow we get from \eqref{eq:EnEq} and \eqref{eq:EstGrad} for all $t \in [0,T)$
$$
  \frac d {dt} \en_{\lambda_1, \lambda_2} (f_t)= -\|\grad_{\lambda_1, \lambda_2} (f_t)\|^2_{L^2 (d\mu)} \leq - \frac {\pi^2 \lambda_1^2} {\en_{0,0}(f_t)} \leq - \frac {\pi^2 \lambda_1^2} {\en_{\lambda_1,\lambda_2}(f_t)}.
$$
Hence the differential inequality
$$
 \frac d {dt} \left(\en_{\lambda_1, \lambda_2} \right)^2 \leq -2 \pi^2 \lambda_1 ^2
$$
holds. The maximal time of smooth existence would thus satisfy
$$
  T \leq  \frac { \left( \en_{\lambda_1,\lambda_2}(f_0) \right)^2 - (4 \pi)^2 }{ 2 \pi^2 \lambda_1^2} 
$$
as for later times $t$ we would have $\en (f_t) \leq \en_{\lambda_1, \lambda_2} (f_t) < 4\pi$, which is impossible.
This concludes the proof for the case that $\lambda_2 =0$.

Let us finally show that in the case $\lambda_2 >0$ every initial embedding with $\en_{\lambda_1, \lambda_2}(f_0) \leq 8 \pi$  and positive volume stays embedded and hence $\vol(f_t)$  stays non-negative. Otherwise there would be a  first time $t_0$ at which a self intersection occurs and hence especially $vol(f_{t_0})\geq 0$, $\ar(f_{t_0})>0$. But then due to the celebrated inequality by Li and Yau \cite{Li1982} we would have  $\en_{0,0}(f_{t_0})\geq 8 \pi$ which would imply
$$
 \en_{\lambda_1, \lambda_2}(f_{t_0}) > 8 \pi \geq \en_{\lambda_1, \lambda_2}(f_0).
$$
This is impossible since the energy monotonically decreases in time.
 This concludes the proof.
\end{proof}

Let us contemplate on the proof of Theorem~\ref{thm:FiniteTime1} a bit further to see what can be saved of it, if we do not assume that $\lambda_2 \vol(f) \geq 0$. Instead we want to use the different scaling of the surface area and the volume that guarantees that, for $\ar(f)$ small, the surface area dominates the inclosed volume.

The proof of Theorem~\ref{thm:FiniteTime1} is based on the
fact that for any immersion $f: \Sigma \rightarrow \mathbb R^3$, $p \in f(\Sigma)$, $f_\alpha = \alpha (f-p)$ we have
\begin{equation*}
   \frac d {d\alpha} \en_{\lambda_1, \lambda_2} (f_\alpha) |_{\alpha=1} = 2\lambda_1 \ar(f)+ 3 \lambda_2 \vol(f)
\end{equation*}
which we have estimated by $2\lambda_1 \ar(f)$ from below under the assumption that $\lambda_2 \vol(f) \geq 0$. If we just assume that $\lambda_1 > 0$, but not $\lambda_2 \geq 0$,
we can estimate, using the isoperimetric inequality $\vol(f) \leq \frac{\ar(f)^{\frac 32}} {6 \pi^{\frac 12}}$, 
\begin{align*}
 \frac d {d\alpha} \en_{\lambda_1, \lambda_2} (f_\alpha) |_{\alpha=1} = 2\lambda_1 \ar(f)+ 3 \lambda_2 \vol(f) \geq \left( 2 \lambda_1 - |\lambda_2| \frac {\ar(f)^{\frac 1 2}}{2 \pi^{\frac 12}} \right) \ar(f) .
\end{align*}
So if $\ar(f) < \varepsilon_1(\lambda_1, \lambda_2)=\frac {\lambda_1^2} {|\lambda_2|^2} 16 \pi$, we can still bound this derivative from below by a positive multiple of $\ar(f)$. Combining this with the estimate \eqref{eq:EstimateGradient} we get
\begin{equation} \label{eq:EstGrad2}
\|\grad_{\lambda_1, \lambda_2} (f) \|^2_{L^2(\mu)} \geq  \frac {\pi^2 \left( \lambda_1 - |\lambda_2| \frac {\ar(f)^{\frac 1 2}}{4 \pi^{\frac 12}} \right)^2}{ \wen(f)}.
\end{equation}

We can only use our control of the energy $\en_{\lambda_1, \lambda_2}$ to bound the surface area of the immersions along the flow. For this we again use the isoperimetric inequality to show
$$
 \en_{\lambda_1, \lambda_2} (f) = \en(f) + \lambda_1 \ar(f) + \lambda_2 \vol(f) \geq \en(f) + (\lambda_1 - |\lambda_2| \frac {\ar(f)^{\frac 1 2 }} {6 \pi^{\frac 1 2}}) \ar(f).
$$
Under the assumption that $\ar (f)\leq \varepsilon_1$ we can estimate this further and get
\begin{equation} \label{eq:EsMu}
 \ar(f) \leq \frac 3 {\lambda_2}\left(\en_{\lambda_1, \lambda_2} (f) -\wen(f) \right) .
\end{equation}
So if we assume that $\en_{\lambda_1, \lambda_2}(f) - \en(f) < \varepsilon_2(\lambda_1, \lambda_2) = \frac { \lambda_1^2}{3 \lambda_2^2} 16 \pi$ we can recover the estimate
$$
 \ar(f) < \frac {\lambda_1 ^2 }{\lambda_2^2} 16 \pi.
$$

Even for the case of negative $\lambda_2<0$ we can show, with the help of the estimates above, the existence of finite time singularities if both $\en_{\lambda_1,\lambda_2}(f_0)$ is close to $4 \pi$ and the area of the initial surface is smaller than $\varepsilon_1(\lambda_1, \lambda_2)$. 

\begin{theorem} \label{thm:FiniteTime3}
  For $\lambda_1 >0, \lambda_2 <0$ let  $f_0: \Sigma \rightarrow \mathbb R^3$ be an immersion of a compact surface $\Sigma$ without boundary satisfying
 \begin{align}
  \ar(f_0) &< \varepsilon_1(\lambda_1, \lambda_2) ,\label{eq:Eps1}\\
  \en_{\lambda_1, \lambda_2}(f_0) & < 4 \pi + \varepsilon_2 (\lambda_1, \lambda_2) \label{eq:Eps2}
 \end{align}
 where as above $\varepsilon_1(\lambda_1, \lambda_2)  = \frac {\lambda_1^2} {\lambda_2^2} 16 \pi$ and $\varepsilon_2 (\lambda_1, \lambda_2)  = \frac {16 \lambda_1^3} {3\lambda_2^2} \pi$.
 Then the maximal time of existence $T$ of the constrained Willmore flow \eqref{eq:ConstrainedWillmoreFlow} with initial data $f_0$ is finite.
\end{theorem}

\begin{proof}
Note that for all times $t \in [0,T)$ for which $\ar(f_t) \leq \varepsilon_1$ inequality \eqref{eq:EsMu} and $\wen(f_t) \geq 4 \pi$ imply
\begin{equation} \label{eq:EsMu+}
\begin{aligned}
  \ar(f_{t}) &\leq \frac 3 {\lambda_1}( \en_{\lambda_1, \lambda_2} (f_{t}) - \wen(f_t)) \leq \frac 3 {\lambda_1} (\en_{\lambda_1, \lambda_2} (f_{0})- 4 \pi)  \\ & < \frac 3 {\lambda_1} \varepsilon_2 (\lambda_1, \lambda_2)
 = \frac{\lambda_1^2}{\lambda_2^2} 16 \pi = \varepsilon_1(\lambda_1, \lambda_2). 
\end{aligned}
 \end{equation}
Let us use this to show that condition \eqref{eq:Eps1} holds for all $t \in [0,T).$ If not, the intermediate value theorem would give us a time $t_0 \in [0,T]$ such that 
$$
 \ar(f_{t_0}) = \varepsilon_1 (\lambda_1, \lambda_2).
$$
But then \eqref{eq:EsMu+} would imply $\ar(f_{t_0})< \varepsilon_1 (\lambda_1, \lambda_2) $.
So \eqref{eq:Eps1} holds for all $t \in [0,T)$ and hence also \eqref{eq:EsMu+}.
 
Using \eqref{eq:EnEq}, \eqref{eq:EstGrad2}, and \eqref{eq:EsMu+}, we get
 \begin{align*}
\frac d {dt} \en_{\lambda_1, \lambda_2 }(f_t)& = -\|\grad_{\lambda_1, \lambda_2} (f) \|^2_{L^2(\mu)} \leq - \frac {\pi^2 \left( \lambda_1 - |\lambda_2| \frac {\ar(f)^{\frac 1 2}}{4 \pi^{\frac 12}} \right)^2}{ \wen(f_t)} 
\\ & \leq - \frac {\pi^2 \left( \lambda_1 - |\lambda_2| \frac {\ar(f)^{\frac 1 2}}{4 \pi^{\frac 12}} \right)^2}{ \en_{\lambda_1, \lambda_2}(f_t)}.
\end{align*}
and hence
\begin{equation*}
 \frac d {dt} (\en_{\lambda_1, \lambda_2 }(f_t))^2 \leq - \pi^2 \left( \lambda_1 - |\lambda_2| \frac {\ar(f)^{\frac 1 2}}{4 \pi^{\frac 12}} \right)^2.
\end{equation*}
Since the term on the right-hand side is bounded away from zero by \eqref{eq:EsMu+}, a singularity must form in finite time.
\end{proof}

\begin{remark}
 Indeed, the assumption $\ar(f_0) \leq \frac {\lambda_1 ^2 } {\lambda_2^2} 16 \pi $ is optimal since, for $\lambda_1 >0$ and $\lambda_2 <0$, any sphere of radius $-\frac {2\lambda_1}{\lambda_2}$ is a critical point of $\en_{\lambda_1, \lambda_2}$ and has area $\frac {\lambda_1 ^2 } {\lambda_2^2} 16 \pi$. In contrast to this, we do not expect $\varepsilon_2$ to be optimal.
\end{remark}

With essentially the same technique we can prove that the condition $\en_{\lambda_1, \lambda_2}(f_0) \leq 8 \pi$ in Theorem~\ref{thm:FiniteTime1} can be weakened to $\en_{\lambda_1, \lambda_2} (f_0)\leq 8 \pi + \min\{\varepsilon_2(\lambda_1, \lambda_2), 8 \pi\}$. But we additionally have to assume that $\ar(f_0)$ is sufficiently small if the inclosed volume has a negative sign. In this case we will furthermore face an additional problem: Even if we assume that $\vol (f_0) >0$ and the initial surface is embedded, the inequality of Li and Yau does not tell us that the surface stays embedded along the flow. 

Instead, we will use that every sphere eversion must have a quadruple point (cf. \cite{Max1981}). So if after some time we have the situation that $f_{t_0}$ is an embedding with $\vol f_{t_0} < 0$ then there must have been a quadruple point before $t_0$, i.e. due to the inequality of Li and Yau there was a time $t \in (0,t_0)$ such that
$$
 \wen(f_t) \geq 16 \pi.
$$
This allows us to prove

\begin{theorem} \label{thm:FiniteTime2}
 For $\lambda_1 >0, \lambda_2 > 0$ let  $f_0: \Sigma \rightarrow \mathbb R^3$ be an embedding of a compact manifold $\Sigma$ without boundary satisfying $\vol(f_0) >0$
 and
 \begin{align}
  \en_{\lambda_1, \lambda_2}(f_0) & < 8 \pi + \min\{\varepsilon_2 (\lambda_1, \lambda_2),8 \pi\} \label{eq:Eps3}.
 \end{align}
% and $\vol(f) >0$ we additionally assume that $\ar(f) < \varepsilon_1(\lambda_1, \lambda_2)$,
% where as above $\varepsilon_1(\lambda_1, \lambda_2)  = \frac {\lambda_1^2} {\lambda_2^2} 16 \pi $ and $\varepsilon_2 (\lambda_1, \lambda_2) = \frac { \lambda_1^3} {3\lambda_2^2} 16\pi$.
 Then the maximal time of existence $T$ of the constrained Willmore flow \eqref{eq:ConstrainedWillmoreFlow} with initial data $f_0$ is finite.
\end{theorem}

\begin{proof}
 We will give different lower bounds for $$\frac d {dt} \en_{\lambda_1, \lambda_2} (f_t ) =2 \frac{\|\grad_{\lambda_1, \lambda_2}(f_t)\|^2_{L^2}}{\en_{\lambda_1, \lambda_2}(f_t)}$$ depending on whether the inclosed volume is non-negative or not. If $\vol(f_t)\geq 0$, we can proceed as in the proof of Theorem~\ref{thm:FiniteTime1} and get
 \begin{equation*} 
  \frac d {dt} \en_{\lambda_1 , \lambda_2}(f_t)) = - \|\grad_{\lambda_1, \lambda_2} (f_t) \|^2_{L^2(\mu)} \leq  \frac {\pi^2 \lambda^2_1}{ \en_{0,0}(f_t)} \leq -\frac {\pi^2 \lambda^2_1}{ \en_{\lambda_1, \lambda_2}(f_t)}.
 \end{equation*}
 Thus
 \begin{equation} \label{eq:Grad1}
  \frac d {dt} (\en_{\lambda_1 , \lambda_2}(f_t)))^2 \leq - 2{\pi^2 \lambda^2_1}.
 \end{equation}

 To deal with the times $t\in [0,T)$ for which $\vol(f_t)<0$, we first observe that 
 for times $t$ for which $\ar(f_t) \leq \varepsilon_1(\lambda_1, \lambda_2)$ and
 $f_t$ is not an embedding, we have due to \eqref{eq:EsMu} and Li and Yau's inequality
 \begin{equation} \label{eq:EsMuBest}
  \ar(f_{t_0}) \leq \frac 3 {\lambda_1} \left( \en_{\lambda_1, \lambda_2} (f_{t_0}) - 8 \pi \right)
  \leq \frac 3 {\lambda_1} \left( \en_{\lambda_1, \lambda_2} (f_{0}) - 8 \pi \right) < \frac{\lambda_1^2}{\lambda_2^2} 16 \pi = \varepsilon_1(\lambda_1, \lambda_2).
 \end{equation}
 
 Furthermore, we note that if for a time $t_0 \in [0,T)$ we have 
 $$
  \ar(f_t) \leq \varepsilon_1(\lambda_1, \lambda_2) \quad \forall t \in [0,t_0] \text{ with } \vol(f_t) <0,
 $$
 then 
 $$
  \wen(f_t) < 16 \pi \quad \forall t \in [0,t_0],
 $$
 since for all $t$ with $\vol(f_t) \geq 0$ we have
 $$
  \wen (f_t) \leq \en_{\lambda_1, \lambda_2}(f_t) \leq  \en_{\lambda_1, \lambda_2}(f_0)<16 \pi
 $$
 and for all $t$ with $\ar(f_t) < \varepsilon_1$ we also get
 $$
  \wen (f_t) \leq \en_{\lambda_1, \lambda_2}(f_t) \leq  \en_{\lambda_1, \lambda_2}(f_0)<16 \pi.
 $$
 But this implies that $f_{t_0}$ is not embedded as otherwise, due to the result of \cite{Max1981}, there would be a time $t \in [0,t_0]$ such that $f_t$ has a quadruple point and hence by the inequality of Li and Yau
 $$
  \wen(f_{t}) \geq 16.
 $$
 Combining theses two observations, we have shown that for times $t_0 \in [0,T)$ with $\ar(f_{t_0}) \leq \varepsilon_1$ and such that $\ar(f_t) \leq \varepsilon_1$ for all $t \in [0,t_0]$ with $\vol (f_t) <0$  we have \eqref{eq:EsMuBest}
 
Let us use this to show that for all times $t \in [0,T]$ with $\vol(f_t)< 0$ we obtain
 \begin{equation} \label{eq:EsMu2}
  \ar(f_t) < \varepsilon_1(\lambda_1, \lambda_2)
 \end{equation}
 be a continuity argument.
 We set 
 \begin{equation*}
  A:=\{ t \in (0,T): \vol(f_t) < 0\},
 \end{equation*}
 where $A$ of course may be the empty set in which case there is nothing to show. 
 
% Let $i \in I$. If $a_i = 0$ then either $\mu(f_{0}) < \varepsilon_1(\lambda_1, \lambda_2)$ or $\vol(f_{0}) = 0$ which would imply by an inequality of Li and Yau that $\en(f_0) \geq 8 \pi$ and hence
% $$
%  \mu(f_0) = \frac 1 {\lambda_0} (\en_{\lambda_1}(f_0) - \en(f_0)) \leq \frac 1 {\lambda_1} \varepsilon_2 = \frac {16} 3 \frac{\lambda_1 ^2}{\lambda_2^2} \pi < \varepsilon_1(\lambda_1, \lambda_2).
% $$
 As $0 \notin A$ we have $\vol(f_{t}) =0$ for all $t \in \partial A$ and hence $f_t$ cannot be an embedding. Thus by \eqref{eq:EsMuBest}
 $$
  \ar(f_{a_i}) <  \varepsilon_1(\lambda_1, \lambda_2).
  $$
  for all $t \in \partial A$.
%  Hence, 
%  $$
%   \ar(f_{a_i}) < \varepsilon_1(\lambda_1, \lambda_2) \quad \forall i \in I.
%  $$
 If \eqref{eq:EsMu2} does not hold on $A$, due to the intermediate value theorem  there would be a first time $t_0 \in A$ with
 $\ar(f_{t_0}) = \frac {\lambda_1 ^2} {\lambda_2 ^2} 16 \pi$. But then the discussion above implies that we have \eqref{eq:EsMuBest}, especially
 $$
  \ar(f_{t_0})  <  \varepsilon_1,
 $$
 contradicting our choice of $t_0$. Thus \eqref{eq:EsMu2} holds for all times $t \in [0,T]$ with $\vol(f_t)< 0$.
 
 From inequality \eqref{eq:EstGrad2} we get
 \begin{equation*} 
  \frac d {dt} \en_{\lambda_1, \lambda_2}(f_t) = \|\grad_{\lambda_1, \lambda_2} (f_t) \|^2_{L^2(\mu)} \leq  - \frac {\pi^2 \left( \lambda_1 - |\lambda_2| \frac {\ar(f_t)^{\frac 1 2}}{4 \pi^{\frac 12}} \right)^2}{ \wen(f_t)}
 \end{equation*}
 and thus 
 \begin{equation} \label{eq:Grad2}
   \frac d {dt} \left(\en_{\lambda_1, \lambda_2} (f_t ) \right)^2 \leq - 2 \pi^2 \left( \lambda_1 - |\lambda_2| \frac {\ar(f_t)^{\frac 1 2}}{4 \pi^{\frac 12}} \right)^2.
 \end{equation}
 Since $\ar(f_t) \leq \varepsilon_1$ for all $t \in A$, inequality \eqref{eq:EsMu} holds for all $t \in A$ and shows that the right hand side in \eqref{eq:EstGrad2} is bounded away from zero for all $t \in A$.

 Hence, the estimates \eqref{eq:Grad1} and \eqref{eq:Grad2} imply that singularities form in finite time. 
\end{proof}
The Theorems~\ref{thm:FiniteTime1}, \ref{thm:FiniteTime3}, and \ref{thm:FiniteTime2} imply Theorem~\ref{thm:FormationOfSingularities}.

%\section{A reverse isoperimetric inequality and the case $\lambda_1 =0$ $\lambda_2 >0$}

\section{Blowup analysis of finite time singularities} \label{sec:Blo}

For the convenience of the reader let us briefly repeat the blowup construction at a singularity in \cite{McCoy2016} which they perform at the beginning of the proof of Theorem~1.4 on page 25. This result extends the construction of a blowup by Kuwert and Sch\"atzle for the Willmore flow \cite{Kuwert2001,Kuwert2004}.

\begin{theorem} \label{thm:Blowup}
There are constants $\varepsilon_0 >0$ and $c_0 >0$ such that the following holds: Let $f: \Sigma \times [0,T) \rightarrow \mathbb R^3$ be the maximal solution to \eqref{eq:ConstrainedWillmoreFlow} with $T < \infty$, i.e. singularities occur in finite time. Then there is a sequence of times $t_j \uparrow T$, of radii $r_j \downarrow 0$ and points $x_j \in \mathbb R^n$
such that the rescaled flows
$$
 f_j: \Sigma \times [0,c_0] \rightarrow \mathbb R^3, f_j(p,t) := \frac 1 {r_j} (f(p,t_j + r_j^4 t)-x_j)
$$
satisfy
$$
\int_{f_j^{-1}(B_1(0))} \|A_{f_j}\|^2 d\mu_{f_j} \geq \varepsilon_0
$$
and converge smoothly locally to a smooth family of proper immersions 
$$
\tilde f : \tilde \Sigma \times [0,c_0] \rightarrow \mathbb R ^3
$$
in the following sense: We can represent 
$$
 f_j  (\phi_j,t) = \tilde f + u_j(\cdot,t)
$$
where
\begin{itemize}
 \item $\phi_j : \tilde f^{-1}(B_j(0)) \rightarrow U_j$ is a diffeomorphism,
 \item $f_j^{-1}(B_R) \subset U_j$ for $j\geq j(R)$,
 \item $u_j \in C^\infty(\tilde \Sigma\times [0,c_0], \mathbb R^n)$ is normal along $\tilde f$,
 \item $\|\nabla^k u_j \|_{L^\infty (\tilde f^{-1}(B_j(0)))} \rightarrow 0$ as $j \rightarrow 0$.
\end{itemize}
\end{theorem}

We will call such a family of immersion $\tilde f$ a blowup limit in the following.

The next theorem shows that possible blowup limits are stationary and parametrize Willmore surfaces. It is an extension of Theorem 4.4 in \cite{McCoy2016}. Again McCoy and Wheeler have shown the result only under the assumption that the energy of the initial surface is close to the Willmore energy of a sphere $4\pi$.

\begin{theorem} \label{thm:BlowupWillmore}
Let $f: \Sigma \times [0,T) \rightarrow \mathbb R^3$ be the maximal solution to \eqref{eq:ConstrainedWillmoreFlow} with $T < \infty$, i.e. singularities occur in finite time.
Then the blowup limit $\tilde f: \tilde \Sigma_\infty \rightarrow \mathbb R^3$ constructed in Theorem \ref{thm:Blowup} does not depend on time and parametrizes a Willmore surface.
\end{theorem}

\begin{proof}
Using that $f$ satisfies equation \eqref{eq:ConstrainedWillmoreFlow} together with
$$
\Delta_{f_j} H_{f_j} +  = r_j^3 \Delta_f H_{f},
$$
$$
H_{f_j} = r_j H_f,
$$
$$
\nu_{f_j} = \nu_j,
$$
and
$$
\partial_t f_j = r_j^3 \partial_t f 
$$
we get from \eqref{eq:ConstrainedWillmoreFlow} that the $f_j$ satisfy
\begin{equation}
\partial_t f_j = \grad (f_j) + (r_j^2 \lambda_1 H_{f_j} + r_j \lambda_2 )\nu_{f_j}.
\end{equation}
Since $f_j$ converges to $\tilde f$ locally smoothly and $r_j \rightarrow 0$, this implies
\begin{equation}
\partial_t \tilde f = \grad (\tilde f).
\end{equation}
As
\begin{align*}
 \int_0^{c_0} &\left(\int_{\Sigma} \|\grad f_j(x,t) + (\lambda_1 r_j^2 H_{f_j}(x,t) +\lambda_2 r_j) \nu_{f_j} \|^2 d\mu_{f_j} \right) dt \\ &=  \int_{t_j}^{t_j+c_0r_j^4} \left(\int_{\Sigma} \|\grad_{\lambda_1, \lambda_2} f(x,t) \|^2 d\mu_{f_j} \right)dt \\
 &= \en_{\lambda_1, \lambda_2}(t_j) -  \en_{\lambda_1, \lambda_2}(t_j+c_0r_j^4) \\ & \rightarrow 0 
\end{align*}
and $r_j \rightarrow 0$ as $j \rightarrow \infty$, we deduce that $\grad \tilde f = 0.$
\end{proof}

Combining Theorem~\ref{thm:BlowupWillmore} with the classification of Willmore spheres due to Bryant \cite{Bryant1984} and the removability of point singularities of Kuwert and Sch\"atzle \cite{Kuwert2004} we get

\begin{corollary} \label{cor:RoundPoints}
 If $f_0: \mathbb S^2 \rightarrow \mathbb R ^3$ is an immersion of a sphere that develops a singularity in finite time under the locally constrained Willmore flow and $\lim_{t\rightarrow T} \en(f_t) < 8 \pi $, then the blowup limit from Theorem~\ref{thm:Blowup} is a round sphere.
\end{corollary}

\begin{proof}
 Let us first assume that $\hat \Sigma$ is compact. Since then the local convergence of the rescaled solution is in fact global, $\tilde \Sigma$ is a topological sphere. So $\tilde f $ is a Willmore sphere with energy below $8 \pi$ and thus is parametrizing a round sphere by the classification result of Bryant \cite{Bryant1984}.
 
 We now lead the case that $\tilde \Sigma$ is not compact to a contradiction as in \cite{Kuwert2004}. We can assume without loss of generality that $0 \notin \hat f (\hat \Sigma)$ since $\tilde f$ is proper. We consider the images of the $f_j$ under the inversion on the standard sphere $I:\mathbb R^3 \setminus \{0\} \rightarrow R^3 \setminus \{0\}, x \mapsto \frac x {|x|^2}$, which is well-defined for large enough $j \in \mathbb N$. The embeddings $\tilde f_j = I \circ f_j$ converge locally smoothly to the embedding $I \circ \tilde f $ in $\mathbb R^3 \setminus \{0\}$ and due to the M\"obius invariance of the Willmore energy
 $$
  \en( I \circ \tilde f ) \leq \liminf_{j \rightarrow \infty}  \en(\tilde f^j ) 
  =  \liminf_{j \rightarrow \infty}  \en( f^j ) < 8 \pi. 
 $$
 The M\"obius invariance of the Willmore energy also implies that $I\circ \tilde f$ is a Willmore surface away from $0$. 
 Due to the point removability result of Kuwert and Sch\"atzle \cite{Kuwert2004}, $\tilde f^\infty$ can be extended to a Willmore sphere of Willmore energy less than $8\pi$. Hence, due to a result of Bryant \cite{Bryant1984}, it must parametrize a round sphere. But this would imply that $\tilde f$ was a plane - which would contradict
 $$
  \int_{\tilde \Sigma}\|A_{\tilde f }\|^2 d\mu_{\tilde f } >0.
 $$
 Hence, $\tilde \Sigma$ must be compact which concludes the proof.
\end{proof}

Corollary~\ref{cor:RoundPoints} implies Theorem~\ref{thm:RoundPoints} and the 
 following extensions of the main result in \cite{McCoy2016}.

\begin{corollary} \label{cor:RoundSphere}
 Let $\lambda_1>0, \lambda_2 \geq 0$ and $f_0: \Sigma \rightarrow \mathbb R^3$ be a closed smoothly embedded surface without boundary satisfying  $\en_{\lambda_1, \lambda_2}(f_0) < 8 \pi.$ Then the constrained Willmore flow with initial data $f_0$ converges to a round point in finite time.
\end{corollary}

\begin{figure}[t] \label{fig:EnergyBound} 
\includegraphics[height=10cm]{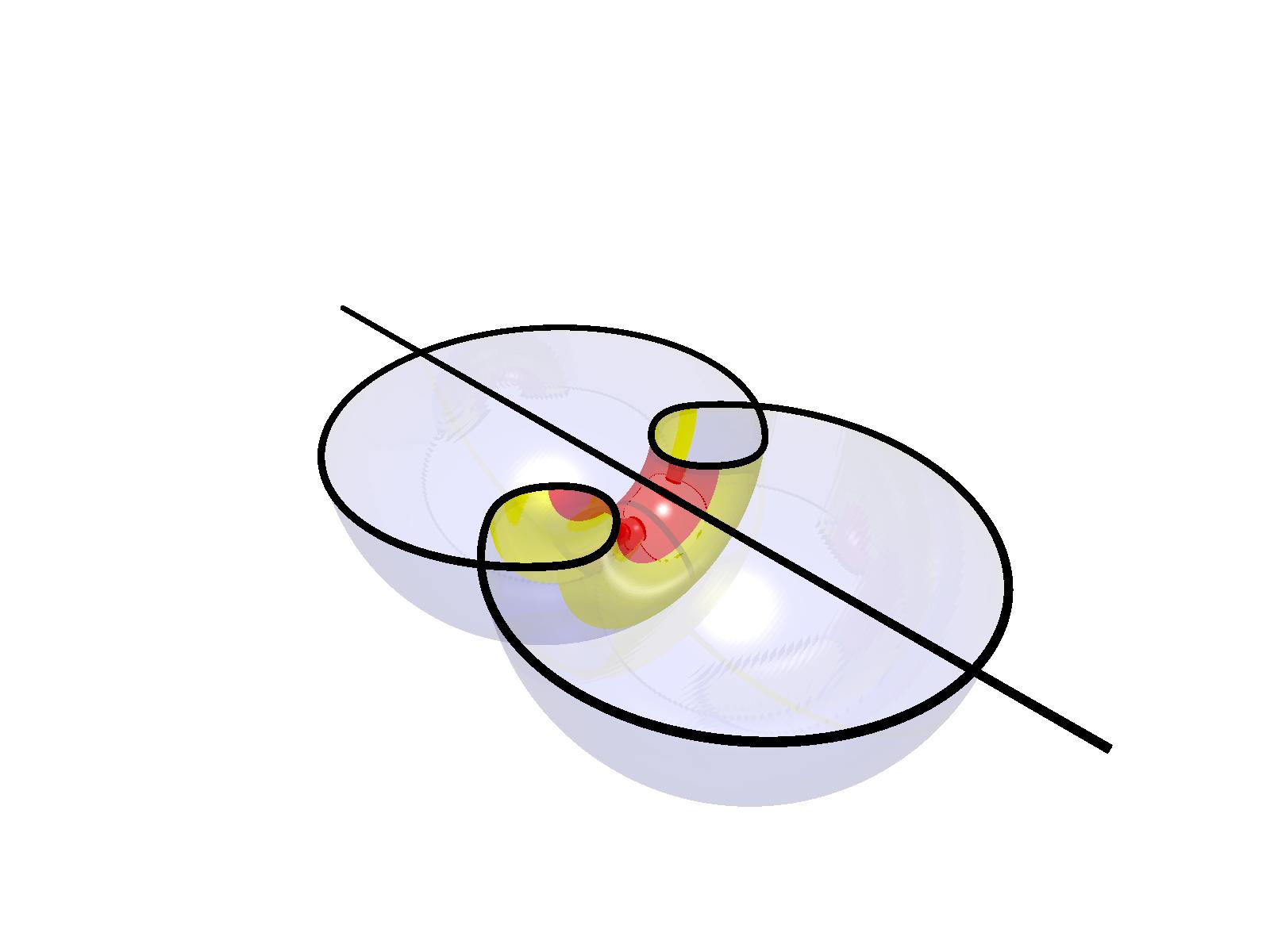}
\caption{ The surface shown above is built out of two
round spheres painted in blue and a piece of a catenoid, painted in red. The yellow part is used
to connect these pieces. One can adapt the parameters such that the Willmore energy of the resulting surface is arbitrary close to $8\pi$.}
\end{figure}

\begin{remark}
The constant $8 \pi$ in Corollary~\ref{cor:RoundSphere} above is sharp which can be shown following the lines of argument in \cite{Blatt2009b}. There we showed that surfaces of revolution exist such that the Gau\ss{} map of the profile curve has index equal to three of Willmore energy just slightly larger than $8\pi$. Figure \ref{fig:EnergyBound} illustrates the construction of such a surface of revolution. As surfaces of revolution remain surfaces of revolution under the flow, the blowup limit can impossibly be a sphere, as otherwise the index of the Gau\ss{} map of the initial surface must have been $\pm 1$. Shrinking the surface if necessary, we see that for every $\varepsilon >0$ we can find an immersed sphere $f:\mathbb S ^2 \rightarrow \mathbb R^3$ with $\en_{\lambda_1, \lambda_2}(f) < 8\pi +  \varepsilon$ that forms a singularity in finite time in such a way that the blowup limit cannot be a round sphere - instead it consists of a finite number of catenoids and planes.
\end{remark}
%\section{A reverse isoperimetric inequality}

%\section{Applications to related evolution equations}

%\bibliographystyle{alpha}

%\bibliography{Master}

\end{document}